\documentclass[a4paper,12pt]{article}

\usepackage{amsmath}
\usepackage{amsfonts}
\usepackage{graphicx}
\usepackage{latexsym, amscd, euscript}
\usepackage{color}
\usepackage{graphics}
\usepackage[all,2cell]{xy}
\usepackage{amssymb}
\usepackage[margin=3.0cm]{geometry}
\usepackage{mathrsfs}
\usepackage{float}
\usepackage{makeidx}
\makeindex
\usepackage{amsmath, amsthm, amsfonts}
\usepackage[compact]{titlesec}
\usepackage{fancyhdr}

\xyoption{import}

\makeindex

\usepackage{amsthm}
\usepackage[dvips,final]{epsfig}
\usepackage[english]{babel}

\titleformat{\chapter}[wrap]
{\bfseries\Huge} {\vspace{-2cm}\filleft\Huge \chaptertitlename
\hspace{.5cm}\thechapter} {2ex} {\titlerule[1pt] \vspace{1pt}
\titlerule
\vspace{1ex}%
\filright}
[\vspace{1ex}%
\titlerule]

\titleformat{\section}[hang]
{\Large\sffamily}
{\thesection.}
{2pt}
{\vspace{2pt}\Large}

\title{\large{\textbf{ALGEBRAIC ELEMENTS OVER THE RING OF POWER SERIES}} \footnotetext{
 keywords: Generalized Puiseux series, Algebraic series, Ring of power series\\
Universidad Espanol, Acapulco de Ju\'arez, Guerrero, M\'exico, saavedravicc@gmail.com }}
 \author{\small{\textbf{V. M. Saavedra}}
}

\date{}
\medskip
\medskip
\begin{document}
\maketitle

\begin{abstract} 

We give a necessary and sufficient condition for a type of generalized power series to be algebraic over the ring of power series with coefficients in a finite field. This result extend a classical theorem of Huang-Stef$ \mathrm{\breve{a}} $nescu.
\end{abstract}

 \newtheorem{thm}{Theorem}[section]
 \newtheorem{coro}[thm]{Corollary}
 \newtheorem{theo}[thm]{Theorem}
 \newtheorem{lema}[thm]{Lemma}
 \newtheorem{prop}[thm]{Proposition}
 \theoremstyle{definition}
 \newtheorem{defi}[thm]{Definition}
 \theoremstyle{remark}
 \newtheorem{rem}[thm]{Remark}
 \newtheorem*{ej}{Example}
 \numberwithin{equation}{section}

 \section{Introduction}
 
 This work is concerning to the problem of describing algebraic elements over the ring of power series.
 It is well known that when $ \Bbbk $ is an algebraically closed field of characteristic zero, the algebraic closure of the field of power series $ \Bbbk((t))$ is the so-called field of Puiseux series. That is, the field formed by the union of all the fields $ \Bbbk((t^{1/d})) $ where $ d $ is a positive integer. This result is known as the Newton-Puiseux theorem (see for example \cite{Br} for a formal presentation). When $ \Bbbk $ is an algebraically closed field of characteristic $ p>0, $ Chevalley \cite{che} noted that the polynomial
 $$ Z^{p}-Z-t^{-1}\in \Bbbk((t))[Z] $$ has no root in the field of Puiseux series. In fact, the Abhyankar$ ^{,} $s paper \cite{Ab} shows that this polynomial can be factored as follows
 
 $$  Z^{p}-Z-t^{-1}=\prod_{i=0}^{p-1}(Z-i-\sum_{j=1}^{\infty}t^{-\frac{1}{p^{j}}}).$$

Using this factorization Huang \cite{H} considered generalized power series of the form 
 
 $$f(t)=\sum_{i } a_{i}t^{\frac{s_{i}}{mp^{n_{i}}}}$$ where $ m\in\mathbb{Z}_{>0}, $ $n_{i}\in\mathbb{Z}_{\geq 0}  $ and $ s_{i}\in\mathbb{Z}$ and the support of $ f $ is a well-ordered set. The set of all generalized power series of this type is a field and Huang \cite{H} proved that contains an algebraic closure of the field $ \Bbbk((t)).$  If $ \Bbbk $ is a perfect field of positive characteristic, the algebraic closure of $ \Bbbk((t)) $ consist of the field of the so-called twist-recurrent series, and it is a result of Kedlaya \cite{K} (see also \cite{K11}). The twist-recurrent series are generalized power series that hold two technical conditions, one over the exponents of the series and another one over the coefficients (see definition in \cite{K}). 
  As example of this type of series are the series that appear in the following theorem given independently by Huang \cite{H} and Stef$ \mathrm{\breve{a}} $nescu \cite{S} (see also \cite{S1}).
  \begin{theo}\label{771}(Huang, Stef$ \breve{a} $nescu) The series $ f(t)=\sum_{i=1 }^{\infty} a_{i}t^{\frac{-1}{p^{i}}}\in \overline{\mathbb{F}_{p}}((t^{\mathbb{Q},p}))$ is algebraic over $ \overline{\mathbb{F}_{p}}((t)) $ if and only if the sequence $ \lbrace a_{i}\rbrace $ is eventually periodic.
  \end{theo}
 
  Theorem \ref{771} can be also deduced from the main result of \cite{K}. S. Vaidya \cite{V} extend the criterion of Huang-Stef$ \mathrm{\breve{a}} $nescu to a certain type of functions, in the case that $ \Bbbk $ is not equal to algebraic closure of its prime field.


 In this paper we present the analogue to Theorem \ref{771} for generalized power series in two variables  with coefficients in a finite field (Corollary \ref{c}).

\section{Generalized power series}
 Let $ \Gamma $ be a totally ordered group and let $\Bbbk $ be a field. The field of Hahn series $ \Bbbk((t^{\Gamma})) $ is defined to be the collection of all elements of the form $$ f=\sum_{\alpha\in \Gamma}c_{\alpha}t^{\alpha}  $$ with $ c_{\alpha}\in \Bbbk $ such that the set of exponents of $ f $ is a well-ordered set. The sum and product are given by 
 
 $$ \sum_{\alpha\in \Gamma}c_{\alpha}t^{\alpha} +\sum_{\alpha\in \Gamma}d_{\alpha}t^{\alpha}=\sum_{\alpha\in \Gamma}(c_{\alpha}+d_{\alpha})t^{\alpha} $$
 
 and $$(\sum_{\alpha^{\prime}\in \Gamma}c_{\alpha^{\prime}}t^{\alpha^{\prime}})(\sum_{\alpha^{\prime\prime}\in \Gamma}d_{\alpha^{\prime\prime}}t^{\alpha^{\prime\prime}})= \sum_{\alpha\in\Gamma}\sum_{\alpha^{\prime}+\alpha^{\prime\prime}=\alpha}c_{\alpha^{\prime}}d_{\alpha^{\prime\prime}}t^{\alpha}.$$
 
 The \textbf{support} of a series $ f $ is the set $ \lbrace \alpha\in \Gamma\mid c_{\alpha}\neq 0   \rbrace.$ The field $ \Bbbk((t^{\Gamma})) $ is also called the field of \textbf{generalized power series }over $\Bbbk $ with support in $ \Gamma. $
It is known that when $ \Bbbk $ is algebraically closed and $ \Gamma $ is a divisible group, the field $ \Bbbk((t^{\Gamma})) $ is algebraically closed \cite{kap}.
 
Rayner noted that we can take the set of series with support in a proper subfamily of the family of all well-ordered subsets of $ \Gamma $ and still lead a field.
He called this subfamily a field-family (see definition in \cite{Ra1}). With the notion of field-family, a family of  of algebraically closed fields of series containing the ring of power series in several variables is given in 
\cite[Theorem 5.3]{kk}.

 Let us denote by $ \Bbbk((t_{1}^{\mathbb{Q}},t_{2}^{\mathbb{Q}})),$ the field of generalized power series with coefficients in $ \Bbbk $ and support a well-ordered subset of $ \mathbb{Q}\times \mathbb{Q}. $
 

 
 
 \begin{lema}\label{L1}
Suppose that $ f_{1},...,f_{l} \in \Bbbk((t_{1}^{\mathbb{Q}},t_{2}^{\mathbb{Q}})) $ with support in $ (-1,0]\times (-1,0] $ and that $ f_{1},...,f_{l} $ are linearly dependent over $\Bbbk((t_{1},t_{2})).$ Then $  f_{1},...,f_{l} $ are also linearly dependent over $\Bbbk.  $
\end{lema}

\begin{proof} Since $ f_{1},...,f_{l} $ are linearly dependent over $\Bbbk((t_{1},t_{2})),$ we can obtain a nonzero linear relation of the form,

 \begin{equation}\label{equ2}
\begin{array}{r@{\hspace{1 pt}} c@{\hspace{1 pt}}c@{\hspace{4pt}}l}
&\varphi_{1}f_{1}+\cdots +\varphi_{l}f_{l}=0,
\end{array}
\end{equation}
where $ \varphi_{k}\in \Bbbk[[t_{1},t_{2}]] $ 
 $ \forall k=1,...,l. $ We can write $ \varphi_{k}=\sum_{n,m}\varphi_{k,(n,m)}t_{1}^{n}t_{2}^{m} $ for some $ \varphi_{k,(n,m)}\in \Bbbk. $ Therefore

 \begin{equation}\label{equ1}
\begin{array}{r@{\hspace{1 pt}} c@{\hspace{1 pt}}c@{\hspace{4pt}}l}
&0= \sum_{n,m}(\sum_{k=1}^{l}\varphi_{k,(n,m)}f_{k})t_{1}^{n}t_{2}^{m}.
\end{array}
\end{equation} 
  
 Now note that the support of $ (\sum_{k=1}^{l}\varphi_{k,(n,m)}f_{k})t_{1}^{n}t_{2}^{m}$ is contained in $ (n-1,n]\times(m-1,m] $ and thus these supports are disjoint for different $ n $ and $ m. $ This implies that the summand in (\ref{equ1}) must be zero for each $ (n,m)$ and thus $\sum_{k=1}^{l}\varphi_{k,(n,m)}f_{k}=0 $ for each $(n,m).$ Note that the $ \varphi_{k,(n,m)} $ cannot all be zero because $ \varphi_{1},..,\varphi_{l} $ would have all been zero. It follows that 
 $ f_{1},...,f_{l} $ are linearly dependent over $\Bbbk.$
 \end{proof}

 From now on $  \mathbb{F} $ will be a finite field.\\

  Denote $ \mathrm{A}_{p}:=\lbrace 0,1,...,p-1\rbrace$ and for $ c $ a nonnegative integer, let $ \mathrm{T}_{c} $ be the subset given by 
$$ \mathrm{T}_{c}:=\bigg \{ -\frac{b_{1}}{p}-\frac{b_{2}}{p^{2}}-\cdots \mid b_{i}\in \mathrm{A}_{p},\sum b_{i}\leq c \bigg \}.$$
 
 We recall that a sequence $ a_{n} $ is \textbf{eventually periodic} if there exist $ s $ and $ m $ such that $ a_{n+s}=a_{n} $ for all $ n\geq m$.\\

 From the proof of Lemma 2.6 in \cite{S}, we can extract the following lemma.

  \begin{lema} \label{even}
  
  Consider a sequence $ \lbrace a_{n} \rbrace \subset \mathbb{F}.$ Suppose that there is $ d, $ and $ c_{0},...,c_{d} $ not all zero, such that $ c_{0}a_{n}+c_{1}a_{n+1}^{p}+\cdots + c_{d}a_{n+d}^{p^{d}}=0,$ $ \forall n\geq k_{0} $ for some $ k_{0}\in \mathbb{Z}_{>0}.$  Then $ \lbrace a_{n} \rbrace $ becomes eventually periodic.
  
   \end{lema}

 \begin{proof} We may suppose that $ c_{d}\neq 0.$  Dividing by the constant $ c_{d} $ we get 
 
 $$ a_{n+d}^{p^{d}}=-c_{d}^{-1}(c_{0}a_{n}+c_{1}a_{n+1}^{p}+\cdots + c_{d-1}a_{n+d-1}^{p^{d-1}})$$

  Thus for every $ n\in\mathbb{Z}_{>0}, $ we have that $  a_{n+d} $ is completely determined by the $ d- $tuple $ (a_{n},a_{n+1},...,a_{n+d-1}).$ Since $a_{n}\in \mathbb{F}\,$  $ \forall n $  and $ \mathbb{F} $ is finite, we get that the set 
   $$ \lbrace  (a_{n},a_{n+1},...,a_{n+d-1})\mid n\in\mathbb{Z}_{>0}\rbrace $$  is finite.
 Therefore there are $ r, t \in\mathbb{Z}_{>0},$ $ r\neq t $ such that $ (a_{r},a_{r+1},...,a_{r+d-1})=
 (a_{t},a_{t+1},...,a_{t+d-1}).$ This implies that $ a_{r+d}^{p^{d}}=a_{t+d}^{p^{d}} $ and then $ a_{r+d}=a_{t+d}.$ But this implies that $ (a_{r+1},...,a_{r+d})= (a_{t+1},...,a_{t+d})$ and then $ a_{r+d+1}=a_{t+d+1}.$ In general it follows that $ a_{r+d+k}=a_{t+d+k}$ for every $ k\in\mathbb{Z}_{\geq 0}.$ We may suppose that $ r<t $ and let $ s:=t-r. $ Then
 $ a_{n+s}=a_{n+t-r}=a_{d+t+n-r-d}.$ Therefore if $ n\geq r+d, $ we get that $ a_{d+t+n-r-d}=a_{d+r+n-r-d}=a_{n}.$ That is,
 $ a_{n+s} =a_{n}$ for every $ n\geq m:=\mathrm{max}(r+d, k_{0}).$
  \end{proof}

   \begin{lema}\label{eti}
     Let $ f=\sum_{i,j}f_{(i,j)}t_{1}^{i}t_{2}^{j}\in \mathbb{F}((t_{1}^{\mathbb{Q}},t_{2}^{\mathbb{Q}})) $ be a series with support in $\mathrm{T}_{c}\times \mathrm{T}_{c}.$ Suppose that there exist positive integers $ M,N $ and $ R $ such that there are $ d_{0},...,d_{RN-1}\in \mathbb{F} $ not all zero such that every sequence $ \lbrace a_{n}\rbrace_{n=0}^{\infty} $ of the form,
$$ a_{n}=f_{(-\frac{\upsilon_{1}}{p^{i_{1}+n}}-\cdots-\frac{\upsilon_{e}}{p^{i_{e}+n}},\quad-\frac{\omega_{1}}{p^{j_{1}+n}}-\cdots-\frac{\omega_{k}}{p^{j_{k}+n}})}$$ satisfies
    \begin{equation}\label{e1}
\begin{array}{r@{\hspace{1 pt}} c@{\hspace{1 pt}}c@{\hspace{4pt}}l}
  d_{0}a_{n}^{p}+d_{1}a_{n+1}^{p^{2}}+\cdots + d_{RN-1}a_{n+RN-1}^{p^{RN}}=0
\end{array}
 \end{equation} for all $ n\geq M.$
If $ v$ and $ p^{RN-1+M}v$ are in $ \mathrm{T}_{c} \times \mathrm{T}_{c},$ then $$ d_{0}f_{p^{RN-1}v}^{p}+d_{1}f_{p^{RN-2}v}^{p^{2}}+\cdots +d_{RN-1}f_{v}^{p^{RN}}=0.$$  
\end{lema}

 \begin{proof}
 Let say that $v:=(-\frac{\upsilon_{1}}{p^{i_{1}}}-\frac{\upsilon_{2}}{p^{i_{2}}}-\cdots-\frac{\upsilon_{e}}{p^{i_{e}}},\quad-\frac{\omega_{1}}{p^{j_{1}}}-\frac{\omega_{2}}{p^{j_{2}}}-\cdots- \frac{\omega_{k}}{p^{j_{k}}})$, therefore
 
$ p^{RN-1+M}v=(-\frac{\upsilon_{1}}{p^{i_{1}-RN+1-M}}-\cdots-\frac{\upsilon_{e}}{p^{i_{e}-RN+1-M}},\quad-\frac{\omega_{1}}{p^{j_{1}-RN+1-M}}-\cdots- \frac{\omega_{k}}{p^{j_{k}-RN+1-M}}).$
 
Consider the sequence $$ a_{n}= f_{(-\frac{\upsilon_{1}}{p^{i_{1}-RN+1-M+n}}-\cdots-\frac{\upsilon_{e}}{p^{i_{e}-RN+1-M+n}},\quad-\frac{\omega_{1}}{p^{j_{1}-RN+1-M+n}}-\cdots- \frac{\omega_{k}}{p^{j_{k}-RN+1-M+n}})}. $$ 
 We get that 
 
  $$ d_{0}f_{p^{RN-1}v}^{p}+d_{1}f_{p^{RN-2}v}^{p^{2}}+\cdots +d_{RN-1}f_{v}^{p^{RN}}= d_{0}a_{M}^{p}+d_{1}a_{M+1}^{p^{2}}+\cdots +d_{RN-1}a_{M+RN-1}^{p^{RN}}=0,$$  by hypothesis.  
 
 \end{proof}

 \section{Algebraic series}

  \begin{lema} \label{Afe} Let $ -\frac{\omega_{1}}{p^{j_{1}}}-\cdots-\frac{\omega_{k}}{p^{j_{k}}} $ be an element of $ \mathrm{T}_{c}. $
 Let  $$ f=\sum_{i} f_{(i,-\frac{\omega_{1}}{p^{j_{1}}}-\cdots-\frac{\omega_{k}}{p^{j_{k}}})}t_{1}^{i}t_{2}^{-\frac{\omega_{1}}{p^{j_{1}}}-\cdots-\frac{\omega_{k}}{p^{j_{k}}}} $$ be a series with support in $  \mathrm{T}_{c}\times  \mathrm{T}_{c}. $ Then $ f $ is algebraic over $ \mathbb{F}((t_{1},t_{2})),$ if and only if every sequence of the form  $ a_{n}=  f_{(-\frac{\upsilon_{1}}{p^{i_{1}}}-\cdots-\frac{\upsilon_{l-1}}{p^{i_{l-1}}}-\frac{1}{p^{n}}(\frac{\upsilon_{l}}{p^{i_{l}}}+\cdots+\frac{\upsilon_{e}}{p^{i_{e}}}),\quad -\frac{\omega_{1}}{p^{j_{1}}}-\cdots-\frac{\omega_{k}}{p^{j_{k}}})}$ is eventually periodic. 
  \end{lema}
 \begin{proof} 
 Suppose that every sequence of the form $$a_{n}=  f_{(-\frac{\upsilon_{1}}{p^{i_{1}}}-\cdots-\frac{\upsilon_{l-1}}{p^{i_{l-1}}}-\frac{1}{p^{n}}(\frac{\upsilon_{l}}{p^{i_{l}}}+\cdots+\frac{\upsilon_{e}}{p^{i_{e}}}),\quad -\frac{\omega_{1}}{p^{j_{1}}}-\cdots-\frac{\omega_{k}}{p^{j_{k}}})} $$ is eventually periodic. Is enough to see that $  f^{\prime}:=\sum f_{(i,-\frac{\omega_{1}}{p^{j_{1}}}-\cdots-\frac{\omega_{k}}{p^{j_{k}}})}t_{1}^{i} $ is algebraic over $ \mathbb{F}((t_{1},t_{2})).$ Note that $ f^{\prime}\in \mathbb{F}((t_{1}^{\mathbb{Q},p})),$ so by \cite[Theorem 15]{K} $ f^{\prime} $ is algebraic over $  \mathbb{F}((t_{1})) $ and then over 
 $\mathbb{F}((t_{1},t_{2})).$\\
  Now suppose that $ f $ is algebraic over $ \mathbb{F}((t_{1},t_{2})).$ Then  $f^{\prime}$ is algebraic over $ \mathbb{F}((t_{1},t_{2})).$ So we can write,
  
  \begin{equation}\label{ebx}
\begin{array}{r@{\hspace{1 pt}} c@{\hspace{1 pt}}c@{\hspace{4pt}}l}
& \varphi_{0}f^{\prime}+\varphi_{1}f^{\prime p}+\cdots +\varphi_{l}f^{\prime p^{l}}=0,
\end{array}
\end{equation}
 for some $ l\in \mathbb{Z}_{\geq 0} $ and for some $ \varphi_{k}\in \mathbb{F}[[t_{1},t_{2}]], $ 
   $ \forall k=0,...,l. $
 Denote $$ m:=\mathrm{min} \lbrace \alpha_{2}\mid \exists \:\alpha_{1} \,\:\mathrm{with} \,\: (\alpha_{1},\alpha_{2})\in \mathrm{supp}(\varphi_{k})\,\: \mathrm{for}\,\: \mathrm{some} \,\: k\in\lbrace 0,...,l\rbrace\rbrace. $$

  Multiply by $t_{2}^{- m} $ both sides of (\ref{ebx}), we get
  
  \begin{equation}\label{a1}
\begin{array}{r@{\hspace{1 pt}} c@{\hspace{1 pt}}c@{\hspace{4pt}}l}
& \varphi_{0}^{\prime}f^{\prime}+\varphi_{1}^{\prime}f^{\prime p}+\cdots +\varphi_{l}^{\prime}f^{\prime p^{l}}=0,
\end{array}
\end{equation}
  
  where some of the $\varphi_{k}^{\prime} $ have some terms just depending on $ t_{1},$ that is, terms with support of the form 
  $ \lbrace (\alpha_{1},0) \rbrace.$ For $ k=0,...,l, $ we can write $$\varphi_{k}^{\prime}=\varphi_{k}^{(1)}+ \varphi_{k}^{(2)}$$   where $ \varphi_{k}^{(1)} $ just contains the terms of $  \varphi_{k}^{\prime} $ with support of the form $ \lbrace (\alpha_{1},0) \rbrace $ and $ \varphi_{k}^{(2)} $ contains the remaining terms.
  By (\ref{a1}), we can write,

  \begin{equation}\label{a2}
\begin{array}{r@{\hspace{1 pt}} c@{\hspace{1 pt}}c@{\hspace{4pt}}l}
& \varphi_{0}^{(1)}f^{\prime}+\cdots +\varphi_{l}^{(1)}f^{\prime p^{l}}=-\varphi_{0}^{(2)}f^{\prime}-\cdots -\varphi_{l}^{(2)}f^{\prime p^{l}}
\end{array}
\end{equation}

  This equality implies that $ \varphi_{0}^{(1)}f^{\prime}+\varphi_{1}^{(1)}f^{\prime p}+\cdots +\varphi_{l}^{(1)}f^{\prime p^{l}}=0, $ because the terms in the left side of (\ref{a2}) just depend on $ t_{1}.$ This means that $ f^{\prime} $ is algebraic over $ \mathbb{F}[[t_{1}]] $ then again by \cite[Theorem 15]{K}, we get that every sequence of the form $ a_{n}$ is eventually periodic. 
  
 \end{proof}

  \begin{lema} \label{AB} Let $ -\frac{\upsilon_{1}}{p^{i_{1}}}-\cdots-\frac{\upsilon_{e}}{p^{i_{e}}} $ be an element of $ \mathrm{T}_{c}. $
 Let  $$ f=\sum_{j} f_{(-\frac{\upsilon_{1}}{p^{i_{1}}}-\cdots-\frac{\upsilon_{e}}{p^{i_{e}}},j)}t_{1}^{-\frac{\upsilon_{1}}{p^{i_{1}}}-\cdots-\frac{\upsilon_{e}}{p^{i_{e}}}}t_{2}^{j} $$ be a series with support in $  \mathrm{T}_{c}\times  \mathrm{T}_{c}. $ Then $ f $ is algebraic over $ \mathbb{F}((t_{1},t_{2})),$ if and only if every sequence of the form $ a_{n}=  f_{(-\frac{\upsilon_{1}}{p^{i_{1}}}-\cdots-\frac{\upsilon_{e}}{p^{i_{e}}} ,\quad -\frac{\omega_{1}}{p^{j_{1}}}-\cdots-\frac{\omega_{r-1}}{p^{j_{r-1}}}-\frac{1}{p^{n}}(\frac{\omega_{r}}{p^{j_{r}}}+\cdots+\frac{\omega_{k}}{p^{j_{k}}}))}$ is eventually periodic. 
  \end{lema}
  
  \begin{proof}
  Apply a similar argument as in the proof of Lemma \ref{Afe}.
 \end{proof}

 \begin{theo}\label{t1}
 
  Let $ f=\sum_{i,j}f_{i,j}t_{1}^{i}t_{2}^{j}\in  \mathbb{F}((t_{1}^{\mathbb{Q}},t_{2}^{\mathbb{Q}})) $ be a series with support in $\mathrm{T}_{c}\times \mathrm{T}_{c}.$ Suppose that there exist positive integers $ M $ and $ N $ such that every sequence $ \lbrace a_{n}\rbrace_{n=0}^{\infty} $ of the form,
$$ a_{n}=f_{(-\frac{\upsilon_{1}}{p^{i_{1}}}-\cdots-\frac{\upsilon_{l-1}}{p^{i_{l-1}}}-\frac{1}{p^{n}}(\frac{\upsilon_{l}}{p^{i_{l}}}+\cdots+\frac{\upsilon_{e}}{p^{i_{e}}}),\quad-\frac{\omega_{1}}{p^{j_{1}}}-\cdots-\frac{\omega_{r-1}}{p^{j_{r-1}}}-\frac{1}{p^{n}}(\frac{\omega_{r}}{p^{j_{r}}}+\cdots+\frac{\omega_{k}}{p^{j_{k}}}))}$$ has period $ N $ after $ M $ terms and the sequences of the form 
$$ b_{n}=f_{(-\frac{\upsilon_{1}}{p^{i_{1}}}-\cdots-\frac{\upsilon_{e}}{p^{i_{e}}},\quad -\frac{\omega_{1}}{p^{j_{1}}}-\cdots-\frac{\omega_{r-1}}{p^{j_{r-1}}}-\frac{1}{p^{n}}(\frac{\omega_{r}}{p^{j_{r}}}+\cdots+\frac{\omega_{k}}{p^{j_{k}}}))}, $$ 
$$ c_{n}=f_{(-\frac{\upsilon_{1}}{p^{i_{1}}}-\cdots-\frac{\upsilon_{l-1}}{p^{i_{l-1}}}-\frac{1}{p^{n}}(\frac{\upsilon_{l}}{p^{i_{l}}}+\cdots+\frac{\upsilon_{e}}{p^{i_{e}}}),\quad -\frac{\omega_{1}}{p^{j_{1}}}-\cdots-\frac{\omega_{k}}{p^{j_{k}}})}$$ are eventually periodic. Then $ f $ is algebraic over 
$ \mathbb{F}((t_{1},t_{2})). $ 
 \end{theo}

\begin{proof}Let $ p^{r}$ be the cardinality of $ \mathbb{F},$  where $ r\in \mathbb{Z}_{\geq 0}.$ There exist $ d_{0},...,d_{2rN-1}\in  \mathbb{F} $ not all zero such that every sequence $ a_{n} $ satisfies
 
  \begin{equation}\label{e2}
\begin{array}{r@{\hspace{1 pt}} c@{\hspace{1 pt}}c@{\hspace{4pt}}l}
  d_{0}a_{n}^{p}+d_{1}a_{n+1}^{p^{2}}+\cdots + d_{2rN-1}a_{n+2rN-1}^{p^{2rN}}=0
\end{array}
 \end{equation}
 
 for all $ n\geq M.$\\
  Indeed, note that 
  
  $  a_{n}^{p}+a_{n+1}^{p^{2}}+\cdots + a_{n+rN-1}^{p^{rN}}-a_{n+rN}^{p^{rN+1}}-\cdots -a_{n+2rN-1}^{p^{2rN}}=0.$ Thus $ d_{s}= 1$ for $ s=0,...,rN-1 $ and $ d_{s}=-1$ for $ s=rN,rN+1,...,2rN-1, $ is a solution.

 Consider one of theses solutions $ (d_{0},...,d_{2rN-1})$ and consider the series 
 
 $$ g:=d_{0}f^{1/p^{2rN-1}}+d_{1}f^{1/p^{2rN-2}}+\cdots+ d_{2rN-1}f. $$

 If $ g=0, $ $ f $ is algebraic. So we can assume that $ g\neq 0.$ We are going to show that $ g $ is a finite sum of algebraic series, from which it follows that $ g^{p^{2rN}} $ is algebraic and thus $ f. $

  Note that the coefficient $ g_{j^{\prime}} $ of $ g $ 
  is
 $$g_{j^{\prime}}=d_{0}f_{p^{2rN-1}j^{\prime}}^{1/p^{2rN-1}}+d_{1}f_{p^{2rN-2}j^{\prime}}^{1/p^{2rN-2}}+\cdots +d_{2rN-1}f_{j^{\prime}},$$
 where $ j^{\prime}\in \mathrm{supp}(f).$

Take
$$j^{\prime}= (-\frac{\upsilon_{1}}{p^{i_{1}^{\prime}}}-\frac{\upsilon_{2}}{p^{i_{2}^{\prime}}}-\cdots-\frac{\upsilon_{e}}{p^{i_{e}^{\prime}}},\quad-\frac{\omega_{1}}{p^{j_{1}^{\prime}}}-\frac{\omega_{2}}{p^{j_{2}^{\prime}}}-\cdots- \frac{\omega_{k}}{p^{j_{k}^{\prime}}})\in \mathrm{supp}(f).$$

With out loss of generality suppose that $ e\leq k.$ Note that there is $ m_{0}\in\mathbb{Z}_{> 0} $ such that $ \frac{b}{p^{n-2rN+1-M}}< \frac{1}{k} $ $ \forall n\geq m_{0} $ and for any $ b\in \lbrace 0,1,2,...,p-1\rbrace.$

 If $ j^{\prime}$  holds that $ i_{1}^{\prime},...,i_{e}^{\prime}\geq m_{0}$ and $ j_{1}^{\prime},...,j_{k}^{\prime}\geq m_{0} $ then $ p^{2rN-1+M}j^{\prime}\in \mathrm{T}_{c}\times \mathrm{T}_{c}$ because 
 
 $$\frac{\upsilon_{1}}{p^{i_{1}^{\prime}-2rN+1-M}}+\cdots+\frac{\upsilon_{e}}{p^{i_{e}^{\prime}-2rN+1-M}}< e\frac{1}{k}\leq 1\;\: \mathrm{and}$$ 
 
  $$ \frac{\omega_{1}}{p^{j_{1}^{\prime}-2rN+1-M}}+\cdots+ \frac{\omega_{k}}{p^{j_{k}^{\prime}-2rN+1-M}}<1.$$

  Thus by Lemma \ref{eti} we can conclude that $ g_{j^{\prime}}^{p^{2rN}}=0 $ and then $ g_{j^{\prime}}=0.$
 We can sort out the remaining terms of $ g $ whose supports use $\lbrace \upsilon_{1},...,\upsilon_{e} \rbrace $ in the first coordinate and $ \lbrace \omega_{1},...,\omega_{k} \rbrace $ in the second coordinate, in series with support of the following forms
  
 \begin{description}
 \item[i)]Series with support of the form $(-\frac{\upsilon_{1}}{p^{i_{1}}}-\cdots-\frac{\upsilon_{e}}{p^{i_{e}}},\quad-\frac{\omega_{1}}{p^{j_{1}}}-\cdots- \frac{\omega_{k}}{p^{j_{k}}}),$ where some of the indices $ i_{s} $ for $ s\in \lbrace 1,...,e\rbrace$ are variable indices (these variable indices are in $\mathbb{Z}_{>0}$) and all the remaining indices are constant and these constant indices satisfy that belongs to the set $  \lbrace 1,...,m_{0}-1\rbrace.$ 
  
 \item[ii)] Series with support of the form $(-\frac{\upsilon_{1}}{p^{i_{1}}}-\cdots-\frac{\upsilon_{e}}{p^{i_{e}}},\quad-\frac{\omega_{1}}{p^{j_{1}}}-\cdots- \frac{\omega_{k}}{p^{j_{k}}}),$ where some of the $ j_{t} $ for $ t\in\lbrace 1,...,k\rbrace$ are variable indices and all the remaining indices are constant and they satisfy that belongs to the set $ \lbrace 1,...,m_{0}-1\rbrace.$
 \item[iii)] Series with support of the form $(-\frac{\upsilon_{1}}{p^{i_{1}}}-\cdots-\frac{\upsilon_{e}}{p^{i_{e}}},\quad-\frac{\omega_{1}}{p^{j_{1}}}-\cdots- \frac{\omega_{k}}{p^{j_{k}}})$  where some of the indices $ i_{s} $ for $ s\in \lbrace 1,...,e\rbrace$ are constant and they satisfy that $ i_{s}\in \lbrace 1,...,m_{0}-1\rbrace$ and where some of the $ j_{t} $ for $ t\in\lbrace 1,...,k\rbrace$ are constant and they satisfy that $ j_{t}\in \lbrace 1,...,m_{0}-1\rbrace$ and the remaining indices are variable indices.
 \end{description}

 Note that there are many finitely series with support as $\mathrm{\textbf{i}}),\mathrm{\textbf{ii}}) $ and 
 $ \mathrm{\textbf{iii}}).$ We are going to show that these series are algebraic.\\

  Take one of this series, let say $ f^{(1)}.$ Suppose that the series $ f^{(1)} $ has some of the indices $ i_{s}$ as variable indices and some of the $ j_{t} $ also as variable indices. Without loss of generality suppose that $ i_{1}^{\prime},...,i_{l}^{\prime} $ and $ j_{1}^{\prime},...,j_{m}^{\prime}$ are the constant indices. Note that it is enough to show that the following series is algebraic,
 $$ \sum g_{(-\frac{\upsilon_{1}}{p^{i_{1}^{\prime}}}-\cdots-\frac{\upsilon_{e}}{p^{i_{e}}},\quad-\frac{\omega_{1}}{p^{j_{1}^{\prime}}}-\cdots- \frac{\omega_{k}}{p^{j_{k}}})}t_{1}^{-\frac{\upsilon_{l+1}}{p^{i_{l+1}}}-\cdots-\frac{\upsilon_{e}}{p^{i_{e}}}}t_{2}^{-\frac{\omega_{m+1}}{p^{j_{m+1}}}-\cdots- \frac{\omega_{k}}{p^{j_{k}}}},$$ where the sum is running over all $ i_{l+1},...,i_{e},j_{m+1},...,j_{k}\in \mathbb{Z}_{> 0} $ such that the $ 2$-tuple

 $$ (-\frac{\upsilon_{1}}{p^{i_{1}^{\prime}}}-\cdots -\frac{\upsilon_{l}}{p^{i_{l}^{\prime}}}-\frac{\upsilon_{l+1}}{p^{i_{l+1}}}-\cdots -\frac{\upsilon_{e}}{p^{i_{e}}},-\frac{\omega_{1}}{p^{j_{1}^{\prime}}}-\cdots -\frac{\omega_{m}}{p^{j_{m}^{\prime}}}-\frac{\omega_{m+1}}{p^{j_{m+1}}}-\cdots -\frac{\omega_{k}}{p^{j_{k}}})\in \mathrm{supp}(g).$$
 
  We denote this series again by $ f^{(1)}.$ Note that sequences of the form
 
 $ e_{n}=g_{(-\frac{\upsilon_{1}}{p^{i_{1}^{\prime}}}-\cdots-\frac{\upsilon_{l}}{p^{i_{l}^{\prime}}}-\frac{1}{p^{n}}(-\frac{\upsilon_{l+1}}{p^{i_{l+1}^{\prime}}}-\cdots-\frac{\upsilon_{e}}{p^{i_{e}^{\prime}}}),\;-\frac{\omega_{1}}{p^{j_{1}^{\prime}}}-\cdots- \frac{\omega_{m}}{p^{j_{m}^{\prime}}}-\frac{1}{p^{n}}(-\frac{\omega_{m+1}}{p^{j_{m+1}^{\prime}}}-\cdots- \frac{\omega_{k}}{p^{j_{k}^{\prime}}})) }$ are eventually periodic because $ e_{n} $ is a sum of sequences of period $ N $ after $ M $ terms, explicitly we have
 $$ e_{n}=d_{2rN-1}f_{(-\frac{\upsilon_{1}}{p^{i_{1}^{\prime}}}-\cdots-\frac{\upsilon_{l}}{p^{i_{l}^{\prime}}}-\frac{1}{p^{n}}(-\frac{\upsilon_{l+1}}{p^{i_{l+1}^{\prime}}}-\cdots-\frac{\upsilon_{e}}{p^{i_{e}^{\prime}}}),\;-\frac{\omega_{1}}{p^{j_{1}^{\prime}}}-\cdots- \frac{\omega_{m}}{p^{j_{m}^{\prime}}}-\frac{1}{p^{n}}(-\frac{\omega_{m+1}}{p^{j_{m+1}^{\prime}}}-\cdots- \frac{\omega_{k}}{p^{j_{k}^{\prime}}})) }+$$
 
    $$d_{2rN-2}f^{1/p}_{(-\frac{\upsilon_{1}}{p^{i_{1}^{\prime}-1}}-\cdots-\frac{\upsilon_{l}}{p^{i_{l}^{\prime}-1}}-\frac{1}{p^{n-1}}(-\frac{\upsilon_{l+1}}{p^{i_{l+1}^{\prime}}}-\cdots-\frac{\upsilon_{e}}{p^{i_{e}^{\prime}}}),\;-\frac{\omega_{1}}{p^{j_{1}^{\prime}-1}}-\cdots- \frac{\omega_{m}}{p^{j_{m}^{\prime}-1}}-\frac{1}{p^{n-1}}(-\frac{\omega_{m+1}}{p^{j_{m+1}^{\prime}}}-\cdots- \frac{\omega_{k}}{p^{j_{k}^{\prime}}})) }+  $$
   
     $$\cdots+d_{0}f^{1/p^{2rN-1}}_{(-\frac{\upsilon_{1}}{p^{i_{1}^{\prime}-(2rN-1)}}-\cdots-\frac{\upsilon_{l}}{p^{i_{l}^{\prime}-(2rN-1)}}-\frac{1}{p^{n-(2rN-1)}}(-\frac{\upsilon_{l+1}}{p^{i_{l+1}^{\prime}}}-\cdots-\frac{\upsilon_{e}}{p^{i_{e}^{\prime}}}),\;\ast),}$$   
  where $ \ast $  is  
 $ -\frac{\omega_{1}}{p^{j_{1}^{\prime}-(2rN-1)}}-\cdots- \frac{\omega_{m}}{p^{j_{m}^{\prime}-(2rN-1)}}-\frac{1}{p^{n-(2rN-1)}}(-\frac{\omega_{m+1}}{p^{j_{m+1}^{\prime}}}-\cdots- \frac{\omega_{k}}{p^{j_{k}^{\prime}}}) $

 At this point, we can repeat all the above argument for $ f^{(1)} $ instead of $ f $ and $ e_{n} $ instead of $ a_{n}.$
 That is, we can find  $ d_{0}^{(1)},...,d_{2rN-1}^{(1)}\in  \mathbb{F}$ not all zero such that 
 
  \begin{equation}\label{e1}
\begin{array}{r@{\hspace{1 pt}} c@{\hspace{1 pt}}c@{\hspace{4pt}}l}
  d_{0}^{(1)}e_{n}^{p}+d_{1}^{(1)}e_{n+1}^{p^{2}}+\cdots + d_{2rN-1}^{(1)}e_{n+2rN-1}^{p^{2rN}}=0
\end{array}
 \end{equation}
 for all $ n\geq M.$ 
  Consider the series $$ g^{(1)}:=d_{0}^{(1)}f^{(1)1/p^{2rN-1}}+d_{1}^{(1)}f^{(1)1/p^{2rN-2}}+\cdots+ d_{2rN-1}^{(1)}f^{(1)}. $$
  If $ g^{(1)}=0 $ then $ f^{(1)} $ is algebraic. So we can assume that $ g^{(1)}\neq 0. $
  
  We are going to show that $ g^{(1)} $ is algebraic, from this it will follows that $ f^{(1)} $ is algebraic.
 Reasoning analogously as before we can see now that $ g^{(1)} $ is a finite sum of series such that the support of these series hold that at least one of the following is true: the number of variable indices $ i_{s} $ is less than $ e-l $ or the number of variable indices $ j_{t} $ is less than $ k-m. $ Take one of this series, let say $ f^{(2)},$  
 
  $$ f^{(2)}=h(t_{1},t_{2}) \sum g_{(-\frac{\upsilon_{1}}{p^{i_{1}^{\prime}}}-\cdots-\frac{\upsilon_{e}}{p^{i_{e}}},\quad-\frac{\omega_{1}}{p^{j_{1}^{\prime}}}-\cdots- \frac{\omega_{k}}{p^{j_{k}}})}t_{1}^{-\frac{\upsilon_{l+a}}{p^{i_{1+a}}}-\cdots-\frac{\upsilon_{e}}{p^{i_{e}}}}t_{2}^{-\frac{\omega_{m+b}}{p^{j_{m+b}}}-\cdots- \frac{\omega_{k}}{p^{j_{k}}}} ,$$ where $ a>1 $ or $ b>1, $ and the sum is running over all 
$ i_{l+a},...,i_{e},j_{m+b},...,j_{k}\in\mathbb{Z}_{>0}$ and $$h(t_{1},t_{2})= t_{1}^{-\frac{\upsilon_{1}}{p^{i_{1}^{\prime}}}-\cdots-\frac{\upsilon_{l+a-1}}{p^{i_{l+a-1}^{\prime}}}}t_{2}^{-\frac{\omega_{1}}{p^{j_{1}^{\prime}}}-\cdots- \frac{\omega_{m+b-1}}{p^{j_{m+b-1}^{\prime}}}}. $$

 Reasoning as before we have that sequences of the form
 
 $ e_{n}^{\prime}=g_{(-\frac{\upsilon_{1}}{p^{i_{1}^{\prime}}}-\cdots-\frac{\upsilon_{l+a-1}}{p^{i_{l+a-1}^{\prime}}}-\frac{1}{p^{n}}(\frac{\upsilon_{1+a}}{p^{i_{l+a}^{\prime}}}+\cdots+\frac{\upsilon_{e}}{p^{i_{e}^{\prime}}}),\;-\frac{\omega_{1}}{p^{j_{1}^{\prime}}}-\cdots- \frac{\omega_{m+b-1}}{p^{j_{m+b-1}^{\prime}}}-\frac{1}{p^{n}}(\frac{\omega_{m+b}}{p^{j_{m+b}^{\prime}}}-\cdots- \frac{\omega_{k}}{p^{j_{k}^{\prime}}})) }$ are eventually periodic.
 Now we can repeat the above argument recursively for $ f^{(2)},$ we will show that 
 
 $$ g^{(2)}:=d_{0}^{(2)}f^{(2)1/p^{2rN-1}}+d_{1}^{(2)}f^{(2)1/p^{2rN-2}}+\cdots+ d_{2rN-1}^{(2)}f^{(2)}$$
 
 is algebraic.
 In general, we get

  $$ g^{(n)}:=d_{0}^{(n)}f^{(n)1/p^{2rN-1}}+d_{1}^{(n)}f^{(n)1/p^{2rN-2}}+\cdots+ d_{2rN-1}^{(n)}f^{(n)},$$
 
 where $  g^{(n)} $ is a finite sum of series and the support of these series holds that at least one of the following is true: the number of variable indices $ i_{s} $ is less than the number of variable indices of the exponents of $ f^{(n)} $ in the first coordinate or the number of variable indices $ j_{t} $ is less than the number of variable indices of the exponents of $ f^{(n)} $ in the second coordinate.
So eventually, in many finitely steps we will get  sequences of the following forms

 $$ f^{(n_{0})}=h(t_{1},t_{2})\sum g_{(-\frac{\upsilon_{1}}{p^{i_{1}^{\prime}}}-\cdots-\frac{\upsilon_{e}}{p^{i_{e}}},\quad-\frac{\omega_{1}}{p^{j_{1}^{\prime}}}-\cdots- \frac{\omega_{k}}{p^{j_{k}}})}t_{1}^{\frac{-\upsilon_{e}}{p^{i_{e}}}}t_{2}^{\frac{-\omega_{k}}{p^{i_{k}}}},$$
 
  where the sum is running over all $ i_{e},j_{k}\in\mathbb{Z}_{>0},$
  and $$  h(t_{1},t_{2})= t_{1}^{-\frac{\upsilon_{1}}{p^{i_{1}^{\prime}}}-\cdots-\frac{\upsilon_{e-1}}{p^{i_{e-1}^{\prime}}}}t_{2}^{-\frac{\omega_{1}}{p^{j_{1}^{\prime}}}-\cdots- \frac{\omega_{k-1}}{p^{j_{k-1}^{\prime}}}}, $$ because on each step the number of variable indices is strictly decreasing. So as before we can get a series

   $$ g^{(n_{0})}:=d_{0}^{(n_{0})}f^{(n_{0})1/p^{2rN-1}}+d_{1}^{(n_{0})}f^{(n_{0})1/p^{2rN-2}}+\cdots+ d_{2rN-1}^{(n_{0})}f^{(n_{0})}. $$ The series $ g^{(n_{0})} $ is a sum of a finite number of series of the following forms
   
   $$ \sum g_{(-\frac{\upsilon_{1}}{p^{i_{1}^{\prime}}}-\cdots-\frac{\upsilon_{e}}{p^{i}},\quad-\frac{\omega_{1}}{p^{j_{1}^{\prime}}}-\cdots- \frac{\omega_{k}}{p^{j_{k}}})}t_{1}^{\frac{-\upsilon_{e}}{p^{i}}}t_{2}^{\frac{-\omega_{k}}{p^{i_{k}}}},$$
 where $ i $ is a constant, or $$ \sum g_{(-\frac{\upsilon_{1}}{p^{i_{1}^{\prime}}}-\cdots-\frac{\upsilon_{e}}{p^{i_{e}}},\quad-\frac{\omega_{1}}{p^{j_{1}^{\prime}}}-\cdots- \frac{\omega_{k}}{p^{j}})}t_{1}^{\frac{-\upsilon_{e}}{p^{i_{e}}}}t_{2}^{\frac{-\omega_{k}}{p^{j}}},$$ where $ j $ is a constant.
 We can apply the hypothesis that sequences as $ (b_{n}) $ and $ (c_{n}) $ are eventually periodic and Lemma \ref{Afe} or Lemma \ref{AB} to conclude that these series are algebraic over $ \mathbb{F}((t_{1},t_{2})).$ It follows that series with support as $ \mathrm{\textbf{i}}),\mathrm{\textbf{ii}}) $ and $ \mathrm{\textbf{iii}}) $ are algebraic over $ \mathbb{F}((t_{1},t_{2})).$  Since there are many finitely elections of the $ \upsilon_{i}^{,} $s and $ \omega_{i}^{,} $s such that $ \sum \upsilon_{i}\leq c$ and $ \sum \omega_{j}\leq c,$ it follows that $ g $ is a finite sum of algebraic series.
Thus $ f $ is algebraic over $ \mathbb{F}((t_{1},t_{2})).$ 
\end{proof}

 Let $ A $ be a ring. Let $ \Gamma $ be an abelian totally ordered group and let $ \infty $ be an element such that $ x< \infty $ for every $ x $ in $ \Gamma $. Extend the law on $ \Gamma\cup \lbrace \infty\rbrace $ by $ \infty+x=\infty+\infty=\infty. $
 A map

  \begin{equation}\label{e22}
\begin{array}{r@{\hspace{1 pt}} c@{\hspace{1 pt}}c@{\hspace{4pt}}l}
 \nu:A\longrightarrow \Gamma \cup \lbrace\infty \rbrace
\end{array}
 \end{equation}
 
 is said to be a valuation of $A $ if satisfies the following properties for all  $ x,y\in A:$
 
\begin{description}
\item[i)]$ \nu(x\cdot y)=\nu(x)+\nu(y),$ 
\item[ii)]$ \nu(x+ y)\geq \mathrm{min}(\nu(x),\nu(y)).$
\item[iii)]$ \nu(x)=\infty $ if and only if $ x=0.$
\end{description}

\begin{rem}\cite[\emph{Remark} 1.3]{Va}\label{Re1} If we have a mapping $ \nu:A\longrightarrow \Gamma \cup \lbrace\infty \rbrace $ with conditions $\mathrm{\textbf{i}}),\mathrm{\textbf{ii}}) $ and with $ \nu(0)=\infty, $ but if we do$\mathrm{n^{,}}$t assume that $ \nu $ takes the value $ \infty $ only for $ 0, $ the set 
 $ \mathcal{P}=\nu^{-1}(\infty)$ is an ideal prime of $ A $ and $ \nu $ induces a valuation on the integral domain $ A/\mathcal{P}. $
\end{rem}


\begin{prop} The second property of a valuation $ \nu $ can be generalized for any set $ \lbrace x_{1},...,x_{n}\rbrace $ in $A$ by 
$\nu(\sum_{i=1}^{n}x_{i}) \geq \mathrm{min}(\nu(x_{1}),...,\nu(x_{n})).$  If the minimum is reached by only one of the
$ \nu(x_{i}) $ we get the equality: $\nu(\sum_{i=1}^{n}x_{i}) = \mathrm{min}(\nu(x_{1}),...,\nu(x_{n}))$.
\end{prop}

\begin{proof}
See for example \cite[\emph{Proposition} 1.3]{Va}
\end{proof}


\begin{theo}\label{u}
 
 Suppose that $ \upsilon_{1},...,\upsilon_{a},$ and $ \omega_{1},...,\omega_{b}$ belong to $ \sum_{p}.$ Let

 $$ f=\sum f_{(-\frac{\upsilon_{1}}{p^{i_{1}}}-\cdots-\frac{\upsilon_{a}}{p^{i_{a}}},-\frac{\omega_{1}}{p^{j_{1}}}-\cdots-\frac{\omega_{b}}{p^{j_{b}}})}t_{1}^{-\frac{\upsilon_{1}}{p^{i_{1}}}-\cdots-\frac{\upsilon_{a}}{p^{i_{a}}}}t_{2}^{-\frac{\omega_{1}}{p^{j_{1}}}-\cdots-\frac{\omega_{b}}{p^{j_{b}}}}$$ 
 
 be a series with support in  $ T_{_{\sum_{i=1}^{a} \nu_{i}}}\times T_{_{\sum_{j=1}^{b} \omega_{j}}}$ and suppose that $ f $ is algebraic over 
$ \mathbb{F}((t_{1},t_{2})).$ Then sequences of the form,

$$ a_{n}=f_{(-\frac{\upsilon_{1}}{p^{i_{1}^{\prime}}}-\cdots-\frac{\upsilon_{l-1}}{p^{i_{l-1}^{\prime}}}-\frac{1}{p^{n}}(\frac{\upsilon_{l}}{p^{i_{l}^{\prime}}}+\cdots+\frac{\upsilon_{a}}{p^{i_{a}^{\prime}}}),\quad-\frac{\omega_{1}}{p^{j_{1}^{\prime}}}-\cdots- \frac{\omega_{r-1}}{p^{j_{r-1}^{\prime}}}-\frac{1}{p^{n}}(\frac{\omega_{r}}{p^{j_{r}^{\prime}}}+\cdots+\frac{\omega_{b}}{p^{j_{b}^{\prime}}}))}$$

$$ b_{n}=f_{(-\frac{\upsilon_{1}}{p^{i_{1}^{\prime}}}-\cdots-\frac{\upsilon_{l-1}}{p^{i_{l-1}^{\prime}}}-\frac{1}{p^{n}}(\frac{\upsilon_{l}}{p^{i_{l}^{\prime}}}+\cdots+\frac{\upsilon_{a}}{p^{i_{a}^{\prime}}}),\quad-\frac{\omega_{1}}{p^{j_{1}^{\prime}}}-\cdots- \frac{\omega_{b}}{p^{j_{b}^{\prime}}})}   $$

and

$$ c_{n}=f_{(-\frac{\upsilon_{1}}{p^{i_{1}^{\prime}}}-\cdots-\frac{\upsilon_{a}}{p^{i_{a}^{\prime}}},\quad-\frac{\omega_{1}}{p^{j_{1}^{\prime}}}-\cdots- \frac{\omega_{r-1}}{p^{j_{r-1}^{\prime}}}-\frac{1}{p^{n}}(\frac{\omega_{r}}{p^{j_{r}^{\prime}}}+\cdots+\frac{\omega_{b}}{p^{j_{b}^{\prime}}}))}, $$

are eventually periodic.  
\end{theo}

\begin{proof}

Consider a sequence $$ a_{n}=f_{(-\frac{\upsilon_{1}}{p^{i_{1}^{\prime}}}-\cdots-\frac{\upsilon_{l-1}}{p^{i_{l-1}^{\prime}}}-\frac{1}{p^{n}}(\frac{\upsilon_{l}}{p^{i_{l}^{\prime}}}+\cdots+\frac{\upsilon_{a}}{p^{i_{a}^{\prime}}}),\quad-\frac{\omega_{1}}{p^{j_{1}^{\prime}}}-\cdots- \frac{\omega_{r-1}}{p^{j_{r-1}^{\prime}}}-\frac{1}{p^{n}}(\frac{\omega_{r}}{p^{j_{r}^{\prime}}}+\cdots+\frac{\omega_{b}}{p^{j_{b}^{\prime}}}))}.$$ Since $ f $ is algebraic over $ \mathbb{F}((t_{1},t_{2})), $ so is $$ g:=t_{1}^{\frac{\upsilon_{1}}{p^{i_{1}^{\prime}}}+\cdots+\frac{\upsilon_{l-1}}{p^{i_{l-1}^{\prime}}}}t_{2}^{\frac{\omega_{1}}{p^{j_{1}^{\prime}}}+\cdots+\frac{\omega_{r-1}}{p^{j_{r-1}^{\prime}}}}f.$$ Since $ g $ is algebraic over $ \mathbb{F}((t_{1},t_{2})), $ the extension $ \mathbb{F}((t_{1},t_{2}))(g)/\mathbb{F}((t_{1},t_{2})) $ is finite. So we can choose any linear dependence among $ g, g^{p}, g^{p^{2}},...,g^{p^{m}},...,.$ Thus there are $ d\in\mathbb{Z}_{>0} $ and $ \varphi_{0},\varphi_{1},...,\varphi_{d}\in \mathbb{F}((t_{1},t_{2})) $ not all zero such that $$ \varphi_{0}g+\varphi_{1}g^{p}+\cdots +\varphi_{d}g^{p^{d}}=0.$$ By clearing denominators, we may suppose that $  \varphi_{0},\varphi_{1},...,\varphi_{d}\in \mathbb{F}[[t_{1},t_{2}]]$.\\

 Now consider the following subseries of $ g:$ 
 
 $$ \overline{g}:= \sum_{n=1}^{\infty} a_{n}t_{1}^{-\frac{\upsilon_{l}}{p^{i_{l}^{\prime}+n}}-\cdots-\frac{\upsilon_{a}}{p^{i_{a}^{\prime}+n}}}t_{2}^{-\frac{\omega_{r}}{p^{j_{r}^{\prime}+n}}-\cdots- \frac{\omega_{b}}{p^{j_{b}^{\prime}+n}}}. $$ 
 
 Let $ h:=g-\overline{g}.$ We have,
$$0=\varphi_{0}g+\varphi_{1}g^{p}+\cdots +\varphi_{d}g^{p^{d}}=\varphi_{0}\overline{g}+\varphi_{1}\overline{g}^{p}+\cdots +\varphi_{d}\overline{g}^{p^{d}}+\varphi_{0}h+\varphi_{1}h^{p}+\cdots +\varphi_{d}h^{p^{d}}.$$
 
 Let $ \psi:=\varphi_{0}\overline{g}+\varphi_{1}\overline{g}^{p}+\cdots +\varphi_{d}\overline{g}^{p^{d}} $ and $ \psi^{\prime}:=\varphi_{0}h+\varphi_{1}h^{p}+\cdots +\varphi_{d}h^{p^{d}}. $ \\

  The exponents of the terms from $ \varphi_{t}\overline{g}^{p^{t}} $ for $ t\in\lbrace 0,1,...,d\rbrace$ are of the form

   \begin{equation}\label{f}
 \begin{array}{r@{\hspace{1 pt}} c@{\hspace{1 pt}}c@{\hspace{4pt}}l}
  (m_{1}-\frac{\upsilon_{l}}{p^{i_{l}^{\prime}+n-t}}-\cdots-\frac{\upsilon_{a}}{p^{i_{a}^{\prime}+n-t}},\quad m_{2}-\frac{\omega_{r}}{p^{j_{r}^{\prime}+n-t}}-\cdots- \frac{\omega_{b}}{p^{j_{b}^{\prime}+n-t}})
 \end{array}
  \end{equation}

  Fix $ t\in\lbrace 0,1,...,d\rbrace.$ Note that there is $ k_{t} $ such that\\
 
  $ 0<\frac{\upsilon_{l}}{p^{i_{l}^{\prime}+n-t}}+\cdots +\frac{\upsilon_{a}}{p^{i_{a}^{\prime}+n-t}}<1  $
and $ 0<\frac{\omega_{r}}{p^{j_{r}^{\prime}+n-t}}+\cdots +\frac{\omega_{b}}{p^{j_{b}^{\prime}+n-t}}<1, $ for all $ n\geq k_{t}. $
 
 We are going to show that the terms from  $ \varphi_{t}\overline{g}^{p^{t}} $ whose exponents satisfy in their expressions (\ref{f}) that $ n\geq k_{t},$ cannot cancel with the terms from $ \varphi_{t}h^{p^{t}}.$

Consider a point in the support of $ \varphi_{t}h^{p^{t}}.$ The first coordinate of this point is of the form  
$$ l_{1}+\frac{\upsilon_{1}}{p^{i_{1}^{\prime}-t}}+\cdots+\frac{\upsilon_{l-1}}{p^{i_{l-1}^{\prime}-t}}-\frac{\upsilon_{1}}{p^{i_{1}-t}}-\cdots-\frac{\upsilon_{a}}{p^{i_{a}-t}} $$
 and the second coordinate is of the form

 $$  l_{2}+\frac{\omega_{1}}{p^{j_{1}^{\prime}-t}}+\cdots+\frac{\omega_{r-1}}{p^{j_{r-1}^{\prime}-t}}-\frac{\omega_{1}}{p^{j_{1}-t}}-\cdots- \frac{\omega_{b}}{p^{j_{b}-t}}.$$
 We will refer to this point as the point $ \mathbf{A}.$  
 
 
 Here the $i_{1}^{\prime},...,i_{l-1}^{\prime}, i_{1},...,i_{a} $ and $j_{1}^{\prime},...,j_{r-1}^{\prime}, j_{1},...,j_{b} $  that appears in $ (3.5) $ satisfy that

 $$(\frac{\upsilon_{1}}{p^{i_{1}^{\prime}}}+\cdots+\frac{\upsilon_{l-1}}{p^{i_{l-1}^{\prime}}}-\frac{\upsilon_{1}}{p^{i_{1}}}-\cdots-\frac{\upsilon_{a}}{p^{i_{a}}},\quad \frac{\omega_{1}}{p^{j_{1}^{\prime}}}+\cdots+\frac{\omega_{r-1}}{p^{j_{r-1}^{\prime}}}-\frac{\omega_{1}}{p^{j_{1}}}-\cdots- \frac{\omega_{b}}{p^{j_{b}}})\in \mathrm{supp}(h).$$  
 
  Fix a point as (\ref{f}) with $ n\geq k_{t} $ (we will refer to this point as the point $ \mathbf{B}$) and the point $ \mathbf{A}.$ Suppose that these points are equal.


   \begin{description}

    \item[i)]
  
 Suppose that $ (3.5) $ holds that $i_{1}^{\prime},...,i_{l-1}^{\prime}, i_{1},\ldots, i_{a}> t  $ and  $ j_{1}^{\prime},...,j_{r-1}^{\prime},j_{1},...,j_{b}> t.$
 
 \begin{itemize}
 \item Suppose that at least one of the following is true: 
 not any numbers $ \frac{\upsilon_{1}}{p^{i_{1}^{\prime}-t}},\cdots,\frac{\upsilon_{l-1}}{p^{i_{l-1}^{\prime}-t}} $ at all are sumands in $ \frac{\upsilon_{1}}{p^{i_{1}-t}}+\cdots+\frac{\upsilon_{a}}{p^{i_{a}-t}} $, or not any numbers $ \frac{\omega_{1}}{p^{j_{1}^{\prime}-t}},\cdots,\frac{\omega_{r-1}}{p^{j_{r-1}^{\prime}-t}}$ at all are sumands in 
 $ \frac{\omega_{1}}{p^{j_{1}-t}}+\cdots+ \frac{\omega_{b}}{p^{j_{b}-t}}.$
 
Without loss of generality suppose that we are in the first case. Note that $$  -1<  \frac{\upsilon_{1}}{p^{i_{1}^{\prime}-t}}+\cdots+\frac{\upsilon_{l-1}}{p^{i_{l-1}^{\prime}-t}}-\frac{\upsilon_{1}}{p^{i_{1}-t}}-\cdots-\frac{\upsilon_{a}}{p^{i_{a}-t}}<1 $$ and $$ -1< \frac{\omega_{1}}{p^{j_{1}^{\prime}-t}}+\cdots+\frac{\omega_{r-1}}{p^{j_{r-1}^{\prime}-t}}-\frac{\omega_{1}}{p^{j_{1}-t}}-\cdots- \frac{\omega_{b}}{p^{j_{b}-t}}<1.$$

 Clearly if $$ \frac{\upsilon_{1}}{p^{i_{1}^{\prime}-t}}+\cdots+\frac{\upsilon_{l-1}}{p^{i_{l-1}^{\prime}-t}}-\frac{\upsilon_{1}}{p^{i_{1}-t}}-\cdots-\frac{\upsilon_{a}}{p^{i_{a}-t}}=0 $$ or $$\frac{\omega_{1}}{p^{j_{1}^{\prime}-t}}+\cdots+\frac{\omega_{r-1}}{p^{j_{r-1}^{\prime}-t}}-\frac{\omega_{1}}{p^{j_{1}-t}}-\cdots- \frac{\omega_{b}}{p^{j_{b}-t}}=0,$$ the points 
 $ \mathbf{A}$ and $ \mathbf{B}$ cannot be equal.\\
 

  If $ \frac{\upsilon_{1}}{p^{i_{1}^{\prime}-t}}+\cdots+\frac{\upsilon_{l-1}}{p^{i_{l-1}^{\prime}-t}}-\frac{\upsilon_{1}}{p^{i_{1}-t}}-\cdots-\frac{\upsilon_{a}}{p^{i_{a}-t}}<0, $
  the equality of the first coordinate of $ \mathbf{A}$ and $ \mathbf{B}$ implies that
 
 

 $$ \frac{\upsilon_{1}}{p^{i_{1}-t}}+\cdots+\frac{\upsilon_{a}}{p^{i_{a}-t}}-\frac{\upsilon_{1}}{p^{i_{1}^{\prime}-t}}-\cdots-\frac{\upsilon_{l-1}}{p^{i_{l-1}^{\prime}-t}}=\frac{\upsilon_{l}}{p^{i_{l}^{\prime}+n-t}}+\cdots+\frac{\upsilon_{a}}{p^{i_{a}^{\prime}+n-t}}$$ and then 
  
  $$ \frac{\upsilon_{1}}{p^{i_{1}-t}}+\cdots+\frac{\upsilon_{a}}{p^{i_{a}-t}}= \frac{\upsilon_{1}}{p^{i_{1}^{\prime}-t}}+\cdots+\frac{\upsilon_{l-1}}{p^{i_{l-1}^{\prime}-t}}+\frac{\upsilon_{l}}{p^{i_{l}^{\prime}+n-t}}+\cdots+\frac{\upsilon_{a}}{p^{i_{a}^{\prime}+n-t}}.$$ Thus we obtain a contradiction.\\

  If $ \frac{\upsilon_{1}}{p^{i_{1}^{\prime}-t}}+\cdots+\frac{\upsilon_{l-1}}{p^{i_{l-1}^{\prime}-t}}-\frac{\upsilon_{1}}{p^{i_{1}-t}}-\cdots-\frac{\upsilon_{a}}{p^{i_{a}-t}}>0, $ we get that 
  
  $$ \frac{\upsilon_{1}}{p^{i_{1}^{\prime}-t}}+\cdots+\frac{\upsilon_{l-1}}{p^{i_{l-1}^{\prime}-t}}-\frac{\upsilon_{1}}{p^{i_{1}-t}}-\cdots-\frac{\upsilon_{a}}{p^{i_{a}-t}}+\frac{\upsilon_{l}}{p^{i_{l}^{\prime}+n-t}}+\cdots+\frac{\upsilon_{a}}{p^{i_{a}^{\prime}+n-t}}=1.$$ Therefore,
  
  $$  \frac{\upsilon_{1}}{p^{i_{1}^{\prime}-t}}+\cdots+\frac{\upsilon_{l-1}}{p^{i_{l-1}^{\prime}-t}}+\frac{\upsilon_{l}}{p^{i_{l}^{\prime}+n-t}}+\cdots+\frac{\upsilon_{a}}{p^{i_{a}^{\prime}+n-t}}=1+\frac{\upsilon_{1}}{p^{i_{1}-t}}+\cdots+\frac{\upsilon_{a}}{p^{i_{a}-t}},$$ a contradiction.

\item Suppose that  $ \frac{\upsilon_{1}}{p^{i_{1}^{\prime}-t}},\cdots,\frac{\upsilon_{l-1}}{p^{i_{l-1}^{\prime}-t}} $ appear as sumands in $ \frac{\upsilon_{1}}{p^{i_{1}-t}}+\cdots+\frac{\upsilon_{a}}{p^{i_{a}-t}} $, and $ \frac{\omega_{1}}{p^{j_{1}^{\prime}-t}},\cdots,\frac{\omega_{r-1}}{p^{j_{r-1}^{\prime}-t}}$  appear as sumands in $ \frac{\omega_{1}}{p^{j_{1}-t}}+\cdots+ \frac{\omega_{b}}{p^{j_{b}-t}}.$
  
  \end{itemize}
  
  Then $$ -\frac{\upsilon_{l}}{p^{i_{l}^{\prime}+n-t}}-\cdots-\frac{\upsilon_{a}}{p^{i_{a}^{\prime}+n-t}}=-\frac{\upsilon_{l}}{p^{i_{l}-t}}-\cdots-\frac{\upsilon_{a}}{p^{i_{a}-t}} $$

 and 
  
  $$ -\frac{\omega_{r}}{p^{j_{r}^{\prime}+n-t}}-\cdots- \frac{\omega_{b}}{p^{j_{b}^{\prime}+n-t}}=-\frac{\omega_{r}}{p^{j_{r}-t}}-\cdots- \frac{\omega_{b}}{p^{j_{b}-t}}.$$

  Then we get 
 $ i_{l}=i_{l}^{\prime}+n,....,i_{a}=i_{a}^{\prime}+n$ and $ j_{r}=j_{r}^{\prime}+n,...,j_{b}=j_{b}^{\prime}+n.$
 
 That is, we get a contradiction because the term with exponent $$  (\frac{\upsilon_{1}}{p^{i_{1}^{\prime}}}+\cdots+\frac{\upsilon_{l-1}}{p^{i_{l-1}^{\prime}}}-\frac{\upsilon_{1}}{p^{i_{1}}}-\cdots-\frac{\upsilon_{a}}{p^{i_{a}}},\quad \frac{\omega_{1}}{p^{j_{1}^{\prime}}}+\cdots+\frac{\omega_{r-1}}{p^{j_{r-1}^{\prime}}}-\frac{\omega_{1}}{p^{j_{1}}}-\cdots- \frac{\omega_{b}}{p^{j_{b}}})$$ belongs to $ h$ and not to $ \overline{g}.$

 It follows that the terms from  $ \varphi_{t}\overline{g}^{p^{t}} $ whose exponents satisfy in their expressions (\ref{f}) that $ n\geq k_{t},$ cannot cancel with the terms from $ \varphi_{t}h^{p^{t}}$ for which the exponents satisfy that 
 $i_{1}^{\prime},...,i_{l-1}^{\prime}, i_{1},\ldots, i_{a}> t  $ and  $ j_{1}^{\prime},...,j_{r-1}^{\prime},j_{1},...,j_{b}> t.$

 \item[ii)]  Now suppose that at least one of the $i_{1}^{\prime},...,i_{l-1}^{\prime}, i_{1},\ldots, i_{a}$
 is less than $ t $  and  $ j_{1}^{\prime},...,j_{r-1}^{\prime},j_{1},...,j_{b}> t.$ 
 
 Without loss of generality suppose that $ i_{1}^{\prime},...,i_{k}^{\prime}< t,$ $ i_{k+1 }^{\prime},...,i_{l-1}^{\prime}, i_{1},\ldots, i_{a}>t $ and  $ j_{1}^{\prime},...,j_{r-1}^{\prime},j_{1},...,j_{b}> t.$ 
 
Now we can apply a similar argument as in the previous case replacing the set formed by $i_{1}^{\prime},...,i_{l-1}^{\prime},$  $ i_{1},\ldots, i_{a} $ by the set formed by  $ i_{k+1 }^{\prime},...,i_{l-1}^{\prime}, i_{1},\ldots, i_{a} $ to get a contradiction.


 \item[iii)] For the case where at least one of the $ j_{1}^{\prime},...,j_{r-1}^{\prime},j_{1},...,j_{b}$ is less than $ t $ 
  and $i_{1}^{\prime},...,i_{l-1}^{\prime}, i_{1},\ldots, i_{a}$ are greater than $ t $ and the case where  at least one of the
   $i_{1}^{\prime},...,i_{l-1}^{\prime}, i_{1},\ldots, i_{a}$ is less than $ t $ and  at least one of the $ j_{1}^{\prime},...,j_{r-1}^{\prime},j_{1},...,j_{b}$ is less than $ t $ we can apply a similar argument as in the previous cases to obtain a contradiction.
  
 \end{description}
Thus the terms from  $ \varphi_{t}\overline{g}^{p^{t}} $ whose exponents satisfy in their expressions (\ref{f}) that $ n\geq k_{t},$ cannot cancel with the terms from $ \varphi_{t}h^{p^{t}}$.\\

  For any $ t\in\lbrace 0,1,...,d\rbrace, $ let $ n_{t}:=t+\mathrm{max}\lbrace k_{t^{\prime}}-t^{\prime}\mid t^{\prime}\in \lbrace 0,...,d\rbrace\rbrace.$ Here $k_{t^{\prime}}  $ is defined as above and we can suppose that $ k_{t^{\prime}}>d. $ We define  $$ G_{t}:= \sum_{n=n_{t}}^{\infty} a_{n}t_{1}^{-\frac{\upsilon_{l}}{p^{i_{l}^{\prime}+n}}-\cdots-\frac{\upsilon_{a}}{p^{i_{a}^{\prime}+n}}}t_{2}^{-\frac{\omega_{r}}{p^{j_{r}^{\prime}+n}}-\cdots- \frac{\omega_{b}}{p^{j_{b}^{\prime}+n}}}.$$ 
  
 Note that for any $ t\in\lbrace 0,1,...,d\rbrace,$ we can write
 
 $$ \varphi_{t}\overline{g}^{p^{t}}= \varphi_{t}a_{1}^{p^{t}}t_{1}^{-\frac{\upsilon_{l}}{p^{i_{l}^{\prime}-(t-1)}}-\cdots-\frac{\upsilon_{a}}{p^{i_{a}^{\prime}-(t-1)}}}t_{2}^{-\frac{\omega_{r}}{p^{j_{r}^{\prime}-(t-1)}}-\cdots- \frac{\omega_{b}}{p^{j_{b}^{\prime}-(t-1)}}}+\cdots +$$

$$+ \varphi_{t}a_{n_{t}-1}^{p^{t}}t_{1}^{-\frac{\upsilon_{l}}{p^{i_{l}^{\prime}-(t-n_{t}+1)}}-\cdots-\frac{\upsilon_{a}}{p^{i_{a}^{\prime}-(t-n_{t}+1)}}}t_{2}^{-\frac{\omega_{r}}{p^{j_{r}^{\prime}-(t-n_{t}+1)}}-\cdots- \frac{\omega_{b}}{p^{j_{b}^{\prime}-(t-n_{t}+1)}}}+ \varphi_{t}G_{t}^{p^{t}}.$$

 Let us write  $$ \psi^{\prime \prime} :=\varphi_{0}G_{0}+ \varphi_{1}G_{1}^{p}+\cdots +\varphi_{d}G_{d}^{p^{d}}.$$  
  We are going to see that the terms from $ \psi^{\prime \prime} $ cannot cancel with the terms from  $ \psi^{\prime}.$

 Note that the terms from $  \varphi_{t}G_{t}^{p^{t}} $ are part of the terms of $  \varphi_{t}\overline{g}^{p^{t}} $ whose exponents hold that $ n\geq k_{t}.$ It follows that the terms from 
  $\varphi_{t}G_{t}^{p^{t}}$ cannot cancel with the terms from $\varphi_{t}h^{p^{t}}.$
  Since the exponents of the terms from $ \varphi_{t}G_{t}^{p^{t}} $ have all the same form for any $ t\in\lbrace 0,...,d\rbrace,$ that is, they are of the form  $$ (m_{1}-\frac{\upsilon_{l}}{p^{i_{l}^{\prime}+m}}-\cdots-\frac{\upsilon_{a}}{p^{i_{a}^{\prime}+m}},\, m_{2}-\frac{\omega_{r}}{p^{j_{r}^{\prime}+m}}-\cdots- \frac{\omega_{b}}{p^{j_{b}^{\prime}+m}}) $$ where $ m\geq \mathrm{max}\lbrace k_{t^{\prime}}-t^{\prime}\mid t^{\prime}\in \lbrace 0,...,d\rbrace\rbrace, $ we obtain that the terms from $ \varphi_{t}G_{t}^{p^{t}} $ cannot cancel with the terms from $ \varphi_{t^{\prime}}h^{p^{t^{\prime}}} $ for any  $ t^{\prime}\in\lbrace 0,...,d\rbrace.$ It follows that the terms from $ \psi^{\prime \prime} $ cannot cancel with the terms from $ \psi^{\prime}. $

We can write $$ 0=\psi^{\prime}+\psi^{\prime \prime}+\varphi_{d}a_{1}^{p^{d}} t_{1}^{-\frac{\upsilon_{l}}{p^{i_{l}^{\prime}-(d-1)}}-\cdots-\frac{\upsilon_{a}}{p^{i_{a}^{\prime}-(d-1)}}}t_{2}^{-\frac{\omega_{r}}{p^{j_{r}^{\prime}-(d-1)}}-\cdots- \frac{\omega_{b}}{p^{j_{b}^{\prime}-(d-1)}}}+ $$ 
 $$+\cdots +[\varphi_{d}a_{d-1}^{p^{d}}+\cdots+ \varphi_{2}a_{1}^{p^{2}} ]t_{1}^{-\frac{\upsilon_{l}}{p^{i_{l}^{\prime}-1}}-\cdots-\frac{\upsilon_{a}}{p^{i_{a}^{\prime}-1}}}t_{2}^{-\frac{\omega_{r}}{p^{j_{r}^{\prime}-1}}-\cdots- \frac{\omega_{b}}{p^{j_{b}^{\prime}-1}}}+$$
 
 $$+[\varphi_{d}a_{d}^{p^{d}}+\cdots+ \varphi_{2}a_{2}^{p^{2}}+\varphi_{1}a_{1}^{p} ]t_{1}^{-\frac{\upsilon_{l}}{p^{i_{l}^{\prime}}}-\cdots-\frac{\upsilon_{a}}{p^{i_{a}^{\prime}}}}t_{2}^{-\frac{\omega_{r}}{p^{j_{r}^{\prime}}}-\cdots- \frac{\omega_{b}}{p^{j_{b}^{\prime}}}}+  $$
 
 $$ +[\varphi_{d}a_{d+1}^{p^{d}}+\cdots+ \varphi_{2}a_{3}^{p^{2}}+\varphi_{1}a_{2}^{p} +\varphi_{0}a_{1}]t_{1}^{-\frac{\upsilon_{l}}{p^{i_{l}^{\prime}+1}}-\cdots-\frac{\upsilon_{a}}{p^{i_{a}^{\prime}+1}}}t_{2}^{-\frac{\omega_{r}}{p^{j_{r}^{\prime}+1}}-\cdots- \frac{\omega_{b}}{p^{j_{b}^{\prime}+1}}}+  $$

 $$ +\cdots +[\varphi_{d}a_{n_{d}-1}^{p^{d}}+\cdots +\varphi_{1}a_{n_{1}-1}^{p}+\varphi_{0}a_{n_{0}-1}] t_{1}^{-\frac{\upsilon_{l}}{p^{i_{l}^{\prime}+d_{0}-1}}-\cdots-\frac{\upsilon_{a}}{p^{i_{a}^{\prime}+d_{0}-1}}}t_{2}^{-\frac{\omega_{r}}{p^{j_{r}^{\prime}+d_{0}-1}}-\cdots- \frac{\omega_{b}}{p^{j_{b}^{\prime}+d_{0}-1}}},$$
 
 where $ d_{0}:=\mathrm{max}\lbrace k_{t^{\prime}}-t^{\prime}\mid t^{\prime}\in \lbrace 0,...,d\rbrace\rbrace$.\\

 We are going to show that terms from $ \psi^{\prime \prime} $ cannot cancel with terms from $ -\psi^{\prime}-  \psi^{\prime \prime}:$

 suppose that $$(m_{1}-\frac{\upsilon_{l}}{p^{i_{l}^{\prime}+m}}-\cdots-\frac{\upsilon_{a}}{p^{i_{a}^{\prime}+m}},\, m_{2}-\frac{\omega_{r}}{p^{j_{r}^{\prime}+m}}-\cdots- \frac{\omega_{b}}{p^{j_{b}^{\prime}+m}}),$$ where $ m\geq d_{0} $ is equal to 
 
 $$(l_{1}-\frac{\upsilon_{l}}{p^{i_{l}^{\prime}+n}}-\cdots-\frac{\upsilon_{a}}{p^{i_{a}^{\prime}+n}},l_{2}-\frac{\omega_{r}}{p^{j_{r}^{\prime}+n}}-\cdots- \frac{\omega_{b}}{p^{j_{b}^{\prime}+n}}),$$ where $  -(d-1)\leq n\leq  d_{0}-1.$ If some of the $ i_{l}^{\prime}+n,..., i_{a}^{\prime}+n$ is $ \leq 0 $ clearly we get a contradiction. Similarly we get a contradiction, if some of the $ j_{r}^{\prime}+n,....,j_{b}^{\prime}+n $ is $ \leq 0.$ 
 
In other case we get that $ m=n\leq d_{0}-1,$ a contradiction. It follows that $\psi^{\prime \prime} =0. $ Since   $G_{0}, G_{1}^{p},...,G_{d}^{p^{d}} $ have support in $ (-1,0]\times (-1,0],$ by Lemma \ref{L1} there are $ c_{0},...,c_{n}\in \mathbb{F} $ not all zero such that $$ c_{0}G_{0}+ c_{1}G_{1}^{p}+\cdots +c_{d}G_{d}^{p^{d}}=0.$$
  
  That is,

  $$  \sum_{n=n_{0}}^{\infty}c_{0} a_{n}t_{1}^{-\frac{\upsilon_{l}}{p^{i_{l}^{\prime}+n}}-\cdots-\frac{\upsilon_{a}}{p^{i_{a}^{\prime}+n}}}t_{2}^{-\frac{\omega_{r}}{p^{j_{r}^{\prime}+n}}-\cdots- \frac{\omega_{b}}{p^{j_{b}^{\prime}+n}}}+ $$  $$+\sum_{n=n_{1}}^{\infty}c_{1} a_{n}^{p}t_{1}^{-\frac{\upsilon_{l}}{p^{i_{l}^{\prime}+n-1}}-\cdots-\frac{\upsilon_{a}}{p^{i_{a}^{\prime}+n-1}}}t_{2}^{-\frac{\omega_{r}}{p^{j_{r}^{\prime}+n-1}}-\cdots- \frac{\omega_{b}}{p^{j_{b}^{\prime}+n-1}}}+$$ 
  
  $$+\cdots + \sum_{n=n_{d}}^{\infty}c_{d} a_{n}^{p^{d}}t_{1}^{-\frac{\upsilon_{l}}{p^{i_{l}^{\prime}+n-d}}-\cdots-\frac{\upsilon_{a}}{p^{i_{a}^{\prime}+n-d}}}t_{2}^{-\frac{\omega_{r}}{p^{j_{r}^{\prime}+n-d}}-\cdots- \frac{\omega_{b}}{p^{j_{b}^{\prime}+n-d}}}=0.  $$

 Note that the coefficient of  $ t_{1}^{-\frac{\upsilon_{l}}{p^{i_{l}^{\prime}+n}}-\cdots-\frac{\upsilon_{a}}{p^{i_{a}^{\prime}+n}}}t_{2}^{-\frac{\omega_{r}}{p^{j_{r}^{\prime}+n}}-\cdots- \frac{\omega_{b}}{p^{j_{b}^{\prime}+n}}} $ for $ n\geq d_{0} $ is

  $$ c_{0}a_{n}+c_{1}a_{n+1}^{p}+\cdots + c_{d}a_{n+d}^{p^{d}}=0.$$
  
  Thus, by Lemma \ref{even} the sequence $ a_{n} $ becomes eventually periodic.\\

  We are going to consider the cases of sequences of the form $ b_{n} $ and $ c_{n}.$  Consider the vector $ \omega=(0,1)$  and the mapping $ \upsilon_{\omega}: \mathbb{F}((t_{1}^{\mathbb{Q}},t_{2}^{\mathbb{Q}}))\rightarrow \mathbb{Q} \cup\lbrace \infty\rbrace $ defined by

  $$  \upsilon_{\omega}(h):= 
   \begin{cases} 
      \mathrm{min}\lbrace \alpha\cdot\omega=\alpha_{2} \mid \alpha=(\alpha_{1},\alpha_{2})\in \mathrm{supp}(h)\rbrace            & \mbox{if the minimum exits,}\, h\neq 0  \\
      \infty  & \mbox{other case } 
   \end{cases}$$  
  
   For $ h\in\mathbb{F}((t_{1}^{\mathbb{Q}},t_{2}^{\mathbb{Q}}))\setminus \upsilon_{\omega}^{-1}(\infty),$ we will denote by $ \mathrm{In}_{\omega}(h),$ the sum of the terms of $ h$ such that the exponents of these terms have second coordinate equal to $  \upsilon_{\omega}(h)$.\\

   Consider the valuation induced by $ \upsilon_{\omega}$ (Remark \ref{Re1}), we will denote it by $ \nu_{\omega}$. Since $ f $ is algebraic over the ring of power series  there is a polynomial 
 $$ P(Z)=\varphi_{d}Z^{p^{d}}+\varphi_{d-1}Z^{p^{d-1}}+\cdots +\varphi_{1}Z+\varphi_{0}\in  \mathbb{F}  [[t_{1},t_{2}]][Z]$$ such that $ P(f)=0. $ So we can write,
 
 \begin{equation}\label{eyy}
 \begin{array}{r@{\hspace{1 pt}} c@{\hspace{1 pt}}c@{\hspace{4pt}}l}
 
 \infty=\nu_{\omega}(0)=\nu_{\omega}(\varphi_{0}+\cdots + \varphi_{d}f^{p^{d}})\geq \mathrm{min}(\nu_{\omega}(\varphi_{0}),\cdots ,\nu_{\omega}(\varphi_{d}f^{p^{d}})) 
 
  \end{array}
  \end{equation}

 From this inequality it follows that the minimum happen at least twice. Then

  \begin{equation}\label{ea}
\begin{array}{r@{\hspace{1 pt}} c@{\hspace{1 pt}}c@{\hspace{4pt}}l}
 \sum_{r\in\Lambda} \mathrm{In}_{\omega}(\varphi_{r})(\mathrm{In}_{\omega}(f))^{p^{r}}=0,
\end{array}
 \end{equation}
where $ \Lambda $ is the set of indices $ r $ such that $ \upsilon_{\omega}(\varphi_{r}f^{p^{r}}) $ is the minimum in (\ref{eyy} ). 
 Note that,
 
  $ \mathrm{In}_{\omega}(f)=\sum_{\alpha_{2}=\nu_{\omega}(f)}f_{\alpha}t_{1}^{\alpha_{1}}t_{2}^{\alpha_{2}}$ and $ \mathrm{In}_{\omega}(\varphi_{r})=t_{2}^{\nu_{\omega}(\varphi_{r})} A_{r}(t_{1}).$
 By the equality (\ref{ea}) we get that $ \sum_{\alpha_{2}=\nu_{\omega}(f)}f_{\alpha}t_{1}^{\alpha_{1}}t_{2}^{\alpha_{2}} $ is algebraic over $  \mathbb{F} ((t_{1},t_{2})). $ Since $ f- \mathrm{In}_{\omega}(f) $ is algebraic over $\mathbb{F}((t_{1},t_{2})),$ we can repeat the above argument for 
 $ f- \mathrm{In}_{\omega}(f) $ and so on and so on  to conclude that series of the form $$ \sum_{i} f_{(i,-\frac{\omega_{1}}{p^{j_{1}}}-\cdots-\frac{\omega_{k}}{p^{j_{k}}})}t_{1}^{i}t_{2}^{-\frac{\omega_{1}}{p^{j_{1}}}-\cdots-\frac{\omega_{k}}{p^{j_{k}}}} $$ are algebraic over $  \mathbb{F}((t_{1},t_{2})).$ Now we can apply Lemma \ref{Afe} to conclude that sequences of the form $ b_{n} $ becomes eventually periodic. For sequences of the form $ c_{n} $ we can use a similar argument as above but now with $ \omega=(1,0). $
  
\end{proof}


By combining Theorem \ref{t1} and Theorem \ref{u} we get the following corollary.  
  \begin{coro}\label{c}
 Let $ \mathbb{F} $ be a finite field and $ s_{1},s_{2}\in\mathrm{ A}_{p}.$ Consider the series $$ f=\sum_{i,j>0} f_{(\frac{-s_{1}}{p^{i}},\frac{-s_{2}}{p^{j}})}t_{1}^{\frac{-s_{1}}{p^{i}}}t_{2}^{\frac{-s_{2}}{p^{j}}} \in  \mathbb{F} ((t_{1}^{\mathbb{Q}},t_{2}^{\mathbb{Q}})).$$ Then $ f $ is algebraic over $ \mathbb{F}((t_{1},t_{2}))$ if and only if there exist positive integers $ M $ and $ N $ such that every sequence of the form $ a_{n}= f_{(\frac{-s_{1}}{p^{i_{0}+n}},\frac{-s_{2}}{p^{j_{0}+n}})}$ where $ i_{0},j_{0}\in\mathbb{Z}_{> 0} $ has period $ N $after $ M $ terms and the sequences of the form $ b_{n}= f_{(\frac{-s_{1}}{p^{i}},\frac{-s_{2}}{p^{n}})}$ and $ c_{n}= f_{(\frac{-s_{1}}{p^{n}},\frac{-s_{2}}{p^{j}})}$ are eventually periodic.
 \end{coro} 


 \end{document}